\documentclass[ECP]{ejpecp} % replace ECP by EJP if needed.  

\SHORTTITLE{Convex Hulls of L\'evy Processes} 

\TITLE{Convex Hulls of L\'evy Processes} % \thanks is optional. Insert line breaks with \\

\AUTHORS{%
  Ilya~Molchanov\footnote{University of Bern, Switzerland.
    \EMAIL{ilya.molchanov,florian.wespi@stat.unibe.ch}}
  \thanks{IM is supported in part by Swiss National Science Foundation grant IZ73Z0\_15229.}
  \and %% remove this line and below if single author
  Florian~Wespi\footnotemark[1]}%AUTHORS
 % Separate items with ;

\AMSSUBJ{60G51 ; 60D05 ; 52A22} % Edit. Separate items with ;
\SUBMITTED{} % Edit.
\ACCEPTED{} % Edit.

\VOLUME{0}
\YEAR{}
\PAPERNUM{0}
\DOI{}

\ABSTRACT{Let $X(t)$, $t\geq0$, be a L\'evy process in $\R^d$ 
  starting at the origin. We study the closed convex hull $Z_s$
  of $\{X(t): 0\leq t\leq s\}$. In
  particular, we provide conditions for the integrability of the
  intrinsic volumes of the random set $Z_s$ and find explicit
  expressions for their means in the case of symmetric $\alpha$-stable
  L\'evy processes. If the process is symmetric and each its
  one-dimensional projection is non-atomic, we establish that the
  origin a.s. belongs to the interior of $Z_s$ for all
  $s>0$. Limit theorems for the convex hull of
  L\'evy processes with normal and stable limits are also obtained.}

\renewcommand{\P}{\mathbf P}
\newcommand{\E}{\mathbf E}
\newcommand{\R}{\mathbb R}
\renewcommand{\SS}{\mathbb S}

\DeclareMathOperator{\conv}{conv}
\DeclareMathOperator{\Int}{Int}
\DeclareMathOperator{\cl}{cl}

\begin{document}

%%%%%%%%%%%%%%%%%%%%%%%%%%%%%%%%%%%%%%%%%%%%%%%%%%%%%%%%%%%%%%%%%%%
%%                                                               %%
%% No need for \maketitle.                                       %%
%%                                                               %%
%%%%%%%%%%%%%%%%%%%%%%%%%%%%%%%%%%%%%%%%%%%%%%%%%%%%%%%%%%%%%%%%%%%

%%%%%%%%%%%%%%%%%%%%%%%%%%%%%%%%%%%%%%%%%%%%%%%%%%%%%%%%%%%%%%%%%%%
%%                                                               %%
%% Please replace what follows by the body of your article       %%
%% (up to the bibliography):                                     %%
%%                                                               %%
%%%%%%%%%%%%%%%%%%%%%%%%%%%%%%%%%%%%%%%%%%%%%%%%%%%%%%%%%%%%%%%%%%%

A number of existing results concern convex hulls of stochastic
processes, especially the Brownian motion and random walks, see
\cite{Eldan2014, Majumdar2010, Vysotsky2015}. In contrary,
considerably less is known about convex hulls of general L\'evy
processes with the exception of some results for symmetric
$\alpha$-stable L\'evy processes in $\R^d$ with $\alpha\in(1,2]$,
see \cite{Kampf2012}.

First, recall some concepts from convex geometry. Denote by
$\mathcal{K}_d$ the family of \emph{convex bodies} (non-empty compact
convex sets) in $\R^d$.  For $K,L\in\mathcal{K}_d$, denote by
$K+L=\{x+y : x\in K, y\in L\}$ their Minkowski sum.  It is known that
the volume $V_d(K+tL)$ is a polynomial in $t\geq0$ of degree $d$, see
\cite[Th.~5.1.7]{Schneider2014}.  The \emph{mixed volumes}
$V(K[j],L[d-j])$, $j=0,\dots,d$, appear as coefficients of this
polynomial, so that
\begin{equation}
  \label{eq:polynomial}
  V_d(K+tL)=\sum_{j=0}^d \binom{d}{j}t^{d-j}V(K[j],L[d-j]), \qquad t\geq0.
\end{equation}
The mixed volume is a function of $d$ arguments, and $K[j]$
stands for its $j$ arguments, all being $K$.  The \emph{intrinsic
  volumes} of a convex body $K$ are normalised mixed volumes
\begin{equation}
  \label{eq:iv}
  V_j(K)=\frac{\binom{d}{j}}{\kappa_{d-j}}V(K[j],B^{d}[d-j]), 
  \qquad j=0,\dots,d,
\end{equation}
where $B^d$ denotes the centred $d$-dimensional unit ball and
$\kappa_{d-j}$ is the $(d-j)$-dimensional 
volume of $B^{d-j}$. In particular, $V_d(K)$ is
the volume (or the Lebesgue measure), $V_{d-1}(K)$ is half the surface
area, $V_{d-2}(K)$ is proportional to the integrated mean curvature,
$V_1(K)$ is proportional to the mean width of $K$, and $V_0(K)=1$, see
\cite[Sec.~4.2,~5.3]{Schneider2014}.  The \emph{Hausdorff metric} between
convex bodies is defined by
\begin{displaymath}
  \rho_H(K,L)=\inf \{r\geq0 : K\subset L+rB^d,\;
  L\subset K+rB^d\}.
\end{displaymath}
The $j$-dimensional volume of the parallelepiped spanned by
$u_1,\dots,u_j\in\R^d$ is denoted by $D_j(u_1,\dots,u_j)$.

Let $X(t)$, $t\geq0$, be a \emph{L\'evy process} in $\R^d$ 
starting at the origin. We are interested in
\begin{displaymath}
  Z_s=\cl\conv\{X(t): 0 \leq t \leq s\},
\end{displaymath}
where $\cl(\cdot)$ denotes the closure and $\conv(\cdot)$ the convex
hull. 
It is easy to
see that $Z_s$ is a random closed convex set, see
\cite{Molchanov2005}. We shortly denote $Z=Z_1$. 
In the case of a Brownian motion, the
expected intrinsic volumes of $Z$ are known, see \cite{Eldan2014}.
The integrability of intrinsic volumes of $Z$ for $X$ being a
symmetric $\alpha$-stable L\'evy process with independent coordinates and the
expected mean width of $Z$ were obtained in
\cite{Kampf2012}.

We start by establishing the integrability of intrinsic volumes in
relation to the properties of the L\'evy measure of the general L\'evy
process and then find explicit expressions for their first moments of
all intrinsic volumes for the case of symmetric $\alpha$-stable L\'evy
processes with $\alpha\in(1,2]$.  For this purpose, we
generalise results on expected random determinants from
\cite{Kabluchko2012,Vitale1991}. 

It is also shown that the origin a.s. belongs to the interior of the convex
hull of symmetric L\'evy processes if the scalar product $\langle
X(t),u\rangle$ is non-atomic for each $u\neq0$. As a
direct consequence, we prove that $X(s)$ a.s. belongs to the interior
of the convex hull of $\{X(t): 0\leq t\leq s\}$.
We also consider expectations of the $L_p$-generalisations of mixed
volumes and prove limit theorems for the scaled $Z_t$ as
$t\to\infty$.

\section{Integrability of the intrinsic volumes}
\label{sec:integr-intr-volum}

The integrability of $V_j(Z_s)^p$ for some $p>0$, $s\geq0$, and all $j=0,\dots,d$
is equivalent to $\E V_d(Z_s+B^d)^p<\infty$.  Indeed, the Steiner
formula \cite[Eq.~(4.1)]{Schneider2014} yields that
\begin{displaymath}
  V_d(Z_s+B^d)=\sum_{j=0}^d\kappa_{d-j}V_j(Z_s).
\end{displaymath}
Thus, the existence of the $p$th moment in the left-hand side is
equivalent to the existence of the $p$th moments of all
(non-negative) summands in the right-hand side. 

The \emph{L\'evy measure} of the process $X(t)$, $t\geq0$, is denoted
by $\nu$. Denote
\begin{displaymath}
  \beta_\nu=\sup\left\{\beta >0 :\;
  \int_{\|x\|> 1} \| x \|^\beta \nu(dx) < \infty \right\}.
\end{displaymath}

\begin{theorem}
  \label{thr:integrability}
  If $0\leq p<\beta_\nu$, then $\E V_j(Z_s)^p<\infty$ for all
  $j=0,\dots,d$ and all $s\geq0$.
\end{theorem}
\begin{proof}
  The result is obvious if
  $X(t)$, $t\geq0$, is a deterministic process, so we exclude this
  case in the following.  The main idea is to split the path of the
  L\'evy process into several parts with integrable volumes of their
  convex hulls. The random variables $T_0=0$ and
  \begin{equation}
    \label{eq:Ti}
    T_i=\inf\{t\geq T_{i-1} : X(t)\notin X(T_{i-1})+B^d\},
    \qquad i\geq1,
  \end{equation}
  form an increasing sequence of stopping times with respect to the
  filtration generated by the process, see \cite[Cor.~8]{Bertoin1996}.
  Since $X(t)$ has unbounded support for each non-trivial L\'{e}vy
  process and any $t>0$ \cite[Th.~24.3]{Sato2005}, these stopping
  times are a.s. finite.  The random variables
  $\tilde{T}_i=T_i-T_{i-1}$, $i\geq1$, are independent identically
  distributed. We set $\tilde{T}_0=0$ and 
  consider the renewal process
  \begin{displaymath}
    N_s=\max\{ k\geq0: \tilde{T}_0+\dots + \tilde{T}_k\leq s \}.
  \end{displaymath}
  It is easy to see that $\E N_s^j<\infty$ for all $j\geq1$. 
  Let $I_k$ be the segment in $\R^d$ with end-points at the origin and
  $X(\min(T_k,s))-X(\min(T_{k-1},s))$, $k\geq1$.
  The segments $I_k$, $k\geq1$, are independent. Denote by $\|I_k\|$
  the length of $I_k$.  It is obvious that $\|I_{k}\| \leq 2\sup_{0\leq
    t \leq s}\| X_{t} \|$.  If $p<\beta_\nu$, then
  $\E\|I_{k}\|^p< \infty$ by \cite[Thms.~25.3,~25.18]{Sato2005}
  and \cite[Prop.~25.4]{Sato2005}.

  Observe that $Z_s\subset B^d$ if $N_s=0$ and otherwise
  \begin{displaymath} 
    Z_s\subseteq I_1+\cdots + I_{N_s}+B^d.
  \end{displaymath}
  Thus, 
  \begin{displaymath}
    V_d(Z_s+B^d)\leq V_d(I_1+\cdots+I_{N_s}+2B^d).
  \end{displaymath}
  The right-hand side equals the linear combination of the mixed
  volumes 
  \begin{displaymath}
    \tilde{V}_j=V((I_1+\dots +I_{N_s})[j],B^d[d-j]), \qquad j=0,\dots,d,
  \end{displaymath}
  of the sets $I_1+\dots+I_{N_s}$ and $B^d$, see
  \cite[Th.~5.1.7]{Schneider2014}. Therefore, it suffices to show that
  $\E[\tilde{V}_j^p1_{N_s\geq1}]<\infty$ for all $j=0,\dots,d$. It is
  obvious that $\tilde{V}_0=V_d(B^d)$ is integrable, so assume
  $j\geq1$. For all $k\leq N_s$ and $N_s \geq 1$, 
  let $[-\zeta_k,\zeta_k]$ be the scaled translate of $I_k$
  that is symmetric with respect to the origin and such that $\|\zeta_k\|=1$.
  It is well known that the mixed volumes are translation
  invariant.  In view of \eqref{eq:iv}, the McMullen--Matheron--Weil formula
  \cite[Eq.~(1.4)]{Lutwak2001} yields
  \begin{align*}
    \E[\tilde{V}_j^p1_{N_s\geq1}]&= \sum_{n=1}^{\infty}
    \E\left[\left(\frac{2^j\kappa_{d-j}}{j!\binom{d}{j}} 
    \sum_{1\leq k_1,\dots,k_j \leq n}  
    \|I_{k_1}\|\cdots \|I_{k_j}\|2^{-j}
    D_j(\zeta_{k_1},\dots,\zeta_{k_j})\right)^p 
    1_{N_s=n} \right].
  \end{align*}
  Since $D_j(\zeta_{k_1},\dots,\zeta_{k_j})\leq1$, we have 
    \begin{align*}
    \E[\tilde{V}_j^p1_{N_s\geq1}]
    &\leq \left(\frac{(d-j)!\kappa_{d-j}}{d!} \right)^p
    \sum_{n=1}^{\infty} \E\left[ \left(
    \sum_{1\leq k_1,\dots,k_j \leq n} 
    \| I_{k_1} \| \cdots \|  I_{k_j}\| \right)^p 1_{N_s=n}\right] \\
    &\leq \left(\frac{(d-j)!\kappa_{d-j}}{d!} \right)^p
    \sum_{n=1}^{\infty}n^{j\max(p,1)}
    \E[\left( \| I_1 \| \cdots \|  I_j\| \right)^p1_{N_s=n}].
  \end{align*}
  Since the lengths of the segments $I_1,\dots,I_j$ are independent, 
  \begin{math}
    \E(\| I_1 \| \cdots \|  I_j\|)^{pr}<\infty,
  \end{math}
  where $r>1$ satisfies $pr<\beta_\nu$.
  Hence, H\"older's inequality yields that
  \begin{displaymath}
    \E[\tilde{V}_j^p1_{N_s\geq1}]\leq 
    \left(\frac{(d-j)!\kappa_{d-j}}{d!} \right)^p
    (\E(\| I_1 \| \cdots \|  I_j\|)^{pr})^{1/r}
    \sum_{n=1}^{\infty} n^{j\max(p,1)} \P\{N_s=n\}^{1/q},
  \end{displaymath}
  where $1/r+1/q=1$. Since all moments of $N_s$ are finite, 
  the series converges. 
\end{proof}

\begin{corollary}
  If $X$ is the Brownian motion, then $\E V_j(Z_s)^p<\infty$ for all
  $p\geq0$, all $j=0,\dots,d$ and all $s\geq0$.
\end{corollary}

An analogue of Theorem~\ref{thr:integrability} holds for random walks.
Let $\{\xi_n,n\geq1\}$ be a sequence of i.i.d.  random vectors in
$\R^d$ and let $S_n=\xi_1+\cdots +\xi_n$, $n\geq1$. Denote by $C_n$
the convex hull of the origin and $S_1,\dots,S_n$. The following
result can be proved similarly to Theorem~\ref{thr:integrability}.

\begin{theorem}
  If $\E \| \xi_1 \|^p < \infty$, then $\E V_j(C_n)^p<\infty$ 
  for all $j=0,\dots,d$.
\end{theorem}

\section{Expected intrinsic volumes}
\label{sec:expect-intr-volum}

Let $X(t)$, $t\geq0$, be a \emph{symmetric $\alpha$-stable} L\'{e}vy process
in $\R^{d}$. In the Gaussian case all moments of $V_j(Z)$,
$j=0,\dots,d$, exist. If $\alpha<2$, then its L\'{e}vy measure is
$\nu(dx)=c\| x \|^{-d-\alpha}$ for a constant $c>0$, so that $\E
V_j(Z)^p<\infty$ for each $p\in[0,\alpha)$.  In this section we
calculate the expected intrinsic volumes of $Z$ assuming that
$\alpha>1$.

Recall that the characteristic function of $X(t)$ can be represented
as 
\begin{equation}
  \label{eq:k-def}
  \E \exp\{\imath\langle X(t),u\rangle\}=\exp\left\{-th(K,u)^{\alpha}\right\},
  \qquad u\in\R^{d},t\geq0,
\end{equation}
where $\langle X(t),u\rangle$ is the scalar product and 
\begin{displaymath}
  h(K,u)=\sup\{\langle x,u \rangle :x\in K  \}, \qquad u\in\R^{d},
\end{displaymath}
is the support function of a convex body $K$ called the
\emph{associated zonoid} of $X(1)$, see
\cite{Kampf2012,Molchanov2009}. Zonoids are convex bodies that are
obtained as limits (with respect to the Hausdorff metric)
of zonotopes, i.e. Minkowski sums of segments. 

Recall that a random compact set $Y$
in $\R^d$ is said to be integrably bounded if $\|Y\|=\sup\{\|u\| :u\in
Y\}$ is integrable. Its \emph{selection expectation} $\E Y$ is
defined as the convex body with support function $\E h(Y,u)$,
$u\in\R^d$, see \cite[Sec.~2.1]{Molchanov2005}. Each zonoid $K$ can be
obtained as the selection expectation of the segment $[0,\xi]$ with a
suitably chosen $\xi$. 

We start by proving an auxiliary result on the expected $j$-dimensional 
volume of a parallelepiped spanned by random vectors $\xi_1,\dots,\xi_j\in\R^d$. 

\begin{theorem}
  \label{thr:parallelepiped}
  Let $j\in\{1,\dots,d\}$. If $\xi_1,\dots,\xi_j\in\R^d$ are 
  independent integrable random vectors, then
  \begin{displaymath}
    V(\E[0,\xi_1],\dots,\E[0,\xi_j],B^d[d-j])=
    \frac{(d-j)!}{d!}\kappa_{d-j} \E D_j(\xi_1,\dots,\xi_j).
  \end{displaymath}   
\end{theorem}
\begin{proof}
  For all $k=1,\dots,j$, and $u\in\R^d$,
  \begin{displaymath}
    h(\E[-\xi_k,\xi_k],u)=\E|\langle \xi_k,u\rangle |=\E\left[\left|\left\langle 
    \xi_k/\|\xi_k\|,u \right\rangle \right| \|\xi_k \| 
    1_{\|\xi_k\|\neq 0}\right].
  \end{displaymath}
  We define the measure $\rho_k$ on Borel sets $A$ in the unit sphere
  $\SS^{d-1}$ by letting
  \begin{displaymath}
    \rho_k(A)=\E\left[\mathbf{1}_{\xi_k/\|\xi_k\|\in
        A}\|\xi_k\|\right].
  \end{displaymath}
  % with density
  % \begin{displaymath}
  %   \frac{d\rho_k}{d\P_k}=\|\xi_k\|,
  % \end{displaymath}
  % where $\P_k$ is the probability measure of $\xi_k$. 
  % As a consequence, we obtain
  Then
  \begin{displaymath}
    h(\E[-\xi_k,\xi_k],u)=\int_{\SS^{d-1}}|\langle v,u\rangle |d\rho_k(v), 
    \qquad u\in\R^d,
  \end{displaymath}
  meaning that $\rho_k$ is the generating measure of
  the zonoid $\E[-\xi_k,\xi_k]$, see \cite[Th.~3.5.3]{Schneider2014}.
  % The set $B^d$ is a zonoid with generating measure
  % $\sigma_{d-1}/(2\kappa_{d-1})$, where $\sigma_{d-1}$ is the
  % spherical Lebesgue measure on $\SS^{d-1}$. 
  Note that
  \begin{displaymath}
     V(\E[0,\xi_1],\dots,\E[0,\xi_j],B^d[d-j])=2^{-j}
     V(\E[-\xi_1,\xi_1],\dots,\E[-\xi_j,\xi_j],B^d[d-j]).
  \end{displaymath}
  By \cite[Th.~5.3.2]{Schneider2014}, 
  \begin{multline*}
    V(\E[0,\xi_1],\dots,\E[0,\xi_j],B^d[d-j])\\=
    \frac{(d-j)!}{d!}\kappa_{d-j}\int_{\SS^{d-1}}\cdots \int_{\SS^{d-1}}
    D_j(u_1,\dots,u_j) d\rho_1(u_1)\cdots d\rho_j(u_j).
  \end{multline*}
  It remains to identify the integral as $\E D_j(\xi_1,\dots,\xi_j)$.
\end{proof}
% We use the fact that
%   \begin{displaymath}
%     D_{k+1}(u_1,\dots,u_{k+1})=D_k(u_1,\dots,u_k)\sin\varphi,
%   \end{displaymath}
%   where $\varphi$ denotes the angle between $u_{k+1}$ and 
%   span$\{u_1,\dots,u_k\}$. Then, the repeated application of 
%   \cite[Lemma~?5.3.2?]{Schneider2014} yields 
%   \begin{align*}
%     V(\E[0,\xi_1],\dots,\E[0,\xi_j],B^d[d-j])&=
%     \frac{(d-j)!\kappa_{d-j}}{d!}\int_{\SS^{d-1}}\dots \int_{\SS^{d-1}}
%     D_j(u_1,\dots,u_j) \\
%     &\times d\rho_1(u_1)\cdots d\rho_j(u_j) \\
%     &=\frac{(d-j)!\kappa_{d-j}}{d!} \E D_j(\xi_1,\dots,\xi_j).
%   \end{align*}
% \end{proof}

Let $M_j$ be a $d\times
j$ matrix composed of $j$ columns being i.i.d. copies of a random
vector $\xi\in\R^d$.

\begin{corollary}
  \label{cor:determinant}
  If $\xi\in\R^d$ is an integrable random vector, then
  \begin{displaymath}
    V_j(\E[0,\xi])=\frac{1}{j!}\E\sqrt{\det M_j^T M_j}, \qquad j=1,\dots,d.
  \end{displaymath}   
\end{corollary}
\begin{proof}
  In view of \eqref{eq:iv},
  \begin{displaymath}
    V_j(\E[0,\xi])=\frac{1}{j!}\E D_j(\xi_1,\dots,\xi_j), \qquad j=1,\dots,d,
  \end{displaymath}
  where $\xi_1,\dots,\xi_j$ are i.i.d. copies of the random vector
  $\xi\in\R^d$. It remains to note that
  \begin{displaymath}
    \E D_j(\xi_1,\dots,\xi_j)=\E\sqrt{\det M_j^T M_j}. \qedhere
  \end{displaymath}
\end{proof}

\begin{theorem}
  \label{thr:intrinsic volumes}
  Let $X(t)$, $t \geq0$, be a symmetric $\alpha$-stable L\'{e}vy process 
  in $\R^d$ with $\alpha >1$. Then, for all $j=1,\dots,d$,
  \begin{displaymath}
    \E V_j(Z)=\frac{\Gamma(1-1/\alpha)^j \Gamma(1/\alpha)^j}
   {\pi^j \Gamma(j/\alpha+1)}V_j(K),
  \end{displaymath}
  where $K$ is the associated zonoid of $X(1)$.
\end{theorem}  
\begin{proof}
  The main idea is to approximate the L\'{e}vy process with the random
  walk $S_i=X(i/n)$, $i=0,\dots,n$. Denote by $C_n$ the convex hull of
  $S_0,S_1,\dots,S_n$. It is shown in \cite{Vysotsky2015} that
  \begin{displaymath}
    \E V_j(C_n)=\frac{1}{j!}\sum_
    {\substack{i_1+ \dots + i_j \leq n \\ i_1,\dots ,i_j \geq 1}}
    \frac{1}{i_1 \cdots i_j}
    \E \sqrt{\left| \det\left(\langle S_{i_m}^{(m)}, S_{i_l}^{(l)}
    \rangle\right)_{m,l=1}^j \right|},
    \qquad j=1,\dots,d,
  \end{displaymath}  
  where $S_n^{(1)}, \dots , S_n^{(d)}$ are i.i.d.  random walks that
  arise from i.i.d. copies $X^{(1)},\dots,X^{(d)}$ of the L\'evy
  process. The determinant is known as the Gram determinant and is
  always non-negative, 
  so that the absolute value can be omitted.

  Since $X(t)$ coincides
  in distribution with $t^{1/\alpha}X(1) $ for any $t > 0$,
  \begin{align*}
    \E \sqrt{\det\left(\langle S_{i_m}^{(m)}, S_{i_l}^{(l)}\rangle\right)_{m,l=1}^j}&=
    \E \sqrt{\det\left(\langle X^{(m)}(i_m/n),
        X^{(l)}(i_l/n)\rangle\right)_{m,l=1}^j} \\
    &=\E \sqrt{\det\left(\langle \left(i_m/n\right)^{1/\alpha} X^{(m)}(1), 
    \left(i_l/n\right)^{1/\alpha} X^{(l)}(1)\rangle\right)_{m,l=1}^j} \\
    &=\frac{(i_1 \cdots i_j)^{1/\alpha}}{n^{j/\alpha}} 
    \E \sqrt{\det\left(\langle X^{(m)}(1), X^{(l)}(1)\rangle\right)_{m,l=1}^j} \\
    &=\frac{(i_1 \cdots i_j)^{1/\alpha}}{n^{j/\alpha}} 
    j! V_j(\E[0,X(1)]),
  \end{align*}  
  where the last equality follows from Corollary~\ref{cor:determinant}. 
  It is shown in \cite{Molchanov2009} that
  \begin{displaymath}
    \E[0,X(1)]=\frac{1}{\pi}\Gamma\left(1-\frac{1}{\alpha}\right)K,
  \end{displaymath}
  where $K$ is the associated zonoid of $X(1)$. 
  Thus, 
  \begin{align*}
    \E V_j(C_n)=n^{-j/\alpha}V_j\Big(\frac{1}{\pi}\Gamma
      \Big(1-\frac{1}{\alpha}\Big)K\Big)
    \sum_{\substack{i_1+ \dots + i_j \leq n \\ i_1,\dots ,i_j \geq 1}}
    \left(i_1 \cdots i_j\right)^{1/\alpha-1}\to \E V_j(Z)\quad
    \text{as}\; n\to\infty.
  \end{align*}
  It follows from the Stolz-Ces\`{a}ro theorem
  \cite[Th.~1.22]{Muresan2009} that 
  \begin{multline*}
    \lim_{n\to\infty}
     n^{-j/\alpha} \sum_{\substack{i_1+ \dots + i_j \leq n 
     \\ i_1,\dots ,i_j \geq 1}} \left(i_1 \cdots
   i_j\right)^{1/\alpha-1}\\ 
    = ((n+1)^{j/\alpha} - n^{j/\alpha})^{-1}
    \left( \sum_{\substack{i_1+ \dots + i_j \leq n+1 
    \\ i_1,\dots ,i_j \geq 1}} \left(i_1 \cdots i_j\right)^{1/\alpha-1}
    -\sum_{\substack{i_1+ \dots + i_j \leq n 
    \\ i_1,\dots ,i_j \geq 1}} \left(i_1 \cdots i_j\right)^{1/\alpha-1}\right).
  \end{multline*}
  Since
  \begin{displaymath}
    \lim_{n\to\infty}\frac{(n+1)^x-n^x}{xn^{x-1}}=1
  \end{displaymath}
  for all $x>0$, it remains to find the limit of 
  \begin{align*}
    \frac{\alpha}{j} n^{-j/\alpha+1} \sum_{\substack{i_1+ \dots + i_j = n+1 
    \\ i_1,\dots ,i_j \geq 1}}
    \left(i_1 \cdots i_j\right)^{1/\alpha-1} 
    &= \frac{\alpha}{j} n^{-j/\alpha+1} \sum_{\substack{i_1+ \dots + i_j = n+1 
    \\ i_1,\dots ,i_j \geq 1}}
    \left(\frac{i_1 \cdots i_j}{n^j}\right)^{1/\alpha-1} n^{j/\alpha-j} 
    \left(\frac{n}{n}\right)^{j-1} \\
    &=\frac{\alpha}{j} \sum_{\substack{i_1+ \dots + i_j = n+1 
    \\ i_1,\dots ,i_j \geq 1}}
    \left(\frac{i_1 \cdots i_j}{n^j}\right)^{1/\alpha-1}
    \left(\frac{1}{n}\right)^{j-1}.
  \end{align*}
  As $n\to\infty$, the limit can be written using the multinomial Beta
  function related to the Dirichlet distribution, see
  \cite[p.~11]{Feng2010}, as
  \begin{displaymath}
    \frac{\alpha}{j}\int_{\substack{t_1+ \dots + t_j = 1 
    \\ t_1,\dots ,t_j \geq 0}}(t_1 \cdots t_{j})^{1/\alpha-1}
    dt_1\dots dt_j=\frac{\alpha \Gamma(1/\alpha)^j}{j \Gamma(j/\alpha)}
    =\frac{\Gamma(1/\alpha)^j}{\Gamma(j/\alpha+1)}.
  \end{displaymath} 
  Finally, 
  \begin{align*}
    \E V_j(Z)&=\frac{\Gamma(1/\alpha)^j}
    {\Gamma(j/\alpha+1)} V_j\left(\frac{1}{\pi}
    \Gamma\left(1-\frac{1}{\alpha}\right)K\right)
  \end{align*}
  and the result follows from the homogeneity of the intrinsic
  volumes. 
\end{proof}  

\begin{remark}
  By the self-similarity property, $Z_s$ coincides in distribution
  with $s^{1/\alpha}Z$ for $s>0$, whence $\E V_j(Z_s)=s^{j/\alpha}\E
  V_j(Z)$.
\end{remark}

\begin{example}
  If $X$ is the standard Brownian motion, then $\alpha=2$ and
  $K=\frac{1}{\sqrt{2}}B^d$, so that we recover the result of
  \cite{Eldan2014} 
  \begin{displaymath}
    \E V_j(Z)=\binom{d}{j}\left(\frac{\pi}{2}\right)^{j/2}
    \frac{\Gamma((d-j)/2+1)}{\Gamma(j/2+1)\Gamma(d/2+1)},\qquad j=1,\dots,d.
  \end{displaymath}
\end{example}

\begin{example}
  \label{ex:spherically-symmetric}
  If $X(1)$ is spherically symmetric, then
  \begin{displaymath}
%    \label{eq:char-function}
    \E \exp\{\imath\langle X(1),u\rangle\}=\exp\left\{-c\|u\|^\alpha \right\}
  \end{displaymath}
  for $c>0$, see  \cite[Th.~14.14]{Sato2005}. Then $K=c^{1/\alpha}B^d$,
  so that
  \begin{displaymath}
    \E V_j(Z)=\binom{d}{j}\frac{\kappa_d}{\kappa_{d-j}} 
    \frac{\Gamma(1-1/\alpha)^j \Gamma(1/\alpha)^j}
   {\pi^j \Gamma(j/\alpha+1)}c^{j/\alpha},\qquad j=1,\dots,d.
  \end{displaymath}
\end{example}

\section{Interior of the convex hull}
\label{sec:interior-convex-hull}

It is well known that the convex hull of the Brownian motion in 
$\R^d$ contains the origin as interior
point with probability 1 for each $s>0$, see \cite{Evans1985}. We
extend this result for symmetric L\'evy processes.

\begin{theorem}
  \label{thr:interior-point}
  Let $X(t)$, $t\geq0$, be a symmetric L\'{e}vy process in $\R^d$,
  such that $\langle X(1),u\rangle$ has a non-atomic distribution for
  each $u\neq0$. Then $\P\{0\in \Int Z_s\}=1$ and $\P\{X(s)\in \Int
  Z_s\}=1$ for each $s>0$.
\end{theorem}
\begin{proof}
  Note that 
  \begin{displaymath}
    \P\{0\in \Int Z_s\}=1-\P\{0\in \partial Z_s\},
  \end{displaymath}
  where $\partial Z_s$ is the topological boundary of $Z_s$.  We
  approximate $Z_s$ with the convex hull
  $C_n=\conv\{X(is/n),i=0,\dots,n\}$ of the random walk embedded in
  the process.  Observe that the sequence of events $\{0\in \partial
  C_n \}$ is decreasing to $\{0\in\partial Z_s\}$.  Denote by $Y_n$
  the number of faces of $C_n$ that contain the origin as a
  vertex. Since $X(t)$ is symmetric and, in view of the imposed
  condition, $\P\{X(t) \in H\}=0$ for any affine hyperplane $H$ in
  $\R^d$ and $t>0$, \cite[Eq.~(14),(15)]{Vysotsky2015} yield that
  \begin{displaymath}
    \E Y_n=2 \sum_{1\leq i_2<\dots<i_d\leq n}
    \frac{(2n-2i_d-1)!!}{i_2(2n-2i_d)!!}\prod_{k=2}^{d-1}
    \frac{1}{i_{k+1}-i_k} \quad \sim \quad
    \frac{2(\log n)^{d-1}}{\sqrt{\pi n}} \quad \text{as}\; n\to\infty.
  \end{displaymath}
  Thus, $\E Y_n \to 0$ as $n\to\infty$, so that 
  \begin{displaymath}
    \P\{0\in \partial Z_s\}=\lim_{n\to\infty}
    \P\{0\in \partial C_n\}=0.\qquad 
  \end{displaymath}
  
  Now we prove the second statement and make the convention that
  $X(0-)=0$. The time reversal of $X(t)$, $t\in[0,s]$, is defined as
  \begin{displaymath}
    \tilde{X}(t)=X((s-t)-)-X(s), \qquad t\in[0,s].
  \end{displaymath}
  It is well known that $\tilde{X}(t)$ coincides in distribution with
  $-X(t)$, see \cite[Sec.~II.1]{Bertoin1996}.  By symmetry of the
  process and the fact that the origin almost surely belongs to the
  interior of the convex hull,  
  \begin{displaymath}
    0\in \Int\conv\{X((s-t)-)-X(s) , \ t\in[0,s]\} \quad \text{a.s.}
  \end{displaymath}
  This relation is equivalent to
  \begin{displaymath}
    X(s)\in \Int\conv\{X((s-t)-), \ t\in[0,s]\}\subseteq 
    \Int Z_s \quad \text{a.s.} \qquad \qedhere
  \end{displaymath}
\end{proof}

% The following result says that for each fixed time $s>0$, 
% $X(s)$ is contained in the interior of $Z_s$.

% \begin{theorem}
%   \label{th:interior-point2}
%   Let $X(t)$, $t\geq0$, be a symmetric L\'{e}vy process in $\R^d$, 
%   where $X(1)$ has full support. Then $\P\{X(s)\in \Int Z_s\}=1$
%   for each $s>0$.
% \end{theorem}
% \begin{proof}
%   We make the convention that $X(0-)=0$. The time reversal of $X(t)$,
%   $t\in[0,s]$, is defined as
%   \begin{displaymath}
%     \tilde{X}(t)=X((s-t)-)-X(s), \qquad t\in[0,s].
%   \end{displaymath}
%   It is well known that $\tilde{X}(t)$ coincides in distribution with
%   $-X(t)$, see \cite[Sec.~II.1]{Bertoin1996}.  By symmetry of the
%   process and Theorem~\ref{thr:interior-point}, 
%   \begin{displaymath}
%     0\in \Int\conv\{X((s-t)-)-X(s) , \ t\in[0,s]\} \quad \text{a.s.}
%   \end{displaymath}
%   This relation is equivalent to
%   \begin{displaymath}
%     X(s)\in \Int\conv\{X((s-t)-), \ t\in[0,s]\}\subseteq 
%     \Int Z_s \quad \text{a.s.} \qquad \qedhere
%   \end{displaymath}
% \end{proof}

The distribution of $\langle X(1),u\rangle$ is non-atomic for each $u\neq0$
if projections of the L\'evy measure $\nu$ on any one-dimensional
linear subspace of $\R^d$ is infinite outside the origin, see 
\cite[Th.~27.4]{Sato2005}, or if $X(1)$ has a non-trivial full-dimensional 
Gaussian component. 
%Alternatively, if the L\'evy measure is infinite and
%absolutely continuous, then $X(t)$ is absolutely continuous for all
%$t$ and so this condition also suffices. 

\section{$L_{p}$-geometry of the convex hull}
\label{sec:l_p-geometry-convex}

For convex bodies $L$ and $M$ that both contain the origin and
$p\in[1,\infty)$, the $L_p$-sum $ L +_p M$ is defined via its support
function as
\begin{displaymath}
  h(L +_p  M,u)^p= h(L,u)^p+ h(M,u)^p,\qquad u\in\R^d.
\end{displaymath}
If $p=1$, one recovers the Minkowski sum.

Based on \eqref{eq:polynomial}, the mixed volume $V (L[d-1], M[1])$
can be defined as
\begin{displaymath}
  V(L[d-1],M[1])=\frac{1}{d}\lim_{t\downarrow 0}\frac{V_d(L+tM)-V_d(L)}{t}.
\end{displaymath}
The $L_p$-generalisation of this mixed volume is defined by 
\begin{displaymath}
  V_p(L,M)=\frac{p}{d}\lim_{t\downarrow 0}\frac{V_d(L+_pt^{1/p}M)-V_d(L)}{t},
\end{displaymath}
where it is required that $L$ and $M$ are convex bodies with the
origin as interior point, see \cite[Eq.~(9.11)]{Schneider2014}. 

We generalise the expression for $\E V(L[d-1],Z[1])$
from \cite{Kampf2012} for the $L_p$-case and symmetric
$\alpha$-stable L\'evy processes in $\R^d$ with $\alpha\in(1,2]$.
The case $\alpha=2$ is considered separately, since the integrability
condition is different and in this case we obtain an explicit expression.

\begin{theorem}
  \label{thr:mixed-volume1}
  Let $L$ be a convex body in $\R^d$ with the origin as an interior
  point and $X(t)$, $t\geq 0$, be a symmetric $\alpha$-stable L\'{e}vy
  process in $\R^d$ with $\alpha\in(1,2)$, such that $\langle X(t),u \rangle$ is 
  nondegenerate for all $u\neq0$. 
  Then, for all $p\in[1,\alpha)$,
  \begin{displaymath}
    \E V_p(L,Z)=\E(\sup_{0\leq t \leq 1} R(t))^p V_p(L,K),
  \end{displaymath}
  where $R(t)$ is a symmetric $\alpha$-stable L\'{e}vy process in $\R$
  with the same parameter $\alpha$ as $X(t)$ and scale parameter 1,
  and $K$ is the associated zonoid of $X(1)$.
\end{theorem}
\begin{proof}
  It follows from the definition of $\alpha$-stability that the 
  one-dimensional L\'{e}vy process $\langle X(t),u \rangle$ is 
  symmetric $\alpha$-stable. In view of \eqref{eq:k-def}, 
  \begin{displaymath}
    \E \exp\{\imath s \langle X(t),u \rangle\}
    =\exp\{-t|s|^{\alpha}h(K,u)^{\alpha}\}.
  \end{displaymath}
  The finite-dimensional distributions of the process $\langle X(t),u
  \rangle$, $t\geq0$, coincide with those of the process
  $h(K,u)R(t)$, $t\geq0$, and
  \begin{align*}
    \E h(Z,u)^p 
    &=\E(\sup\{\langle X(t),u \rangle : t\in[0,1]\})^p \\
    &=h(K,u)^p \E(\sup_{0\leq t \leq 1} R(t))^p.
  \end{align*}
  The last expectation is finite if $p<\alpha$.  

  It is shown \cite[Eq.~(9.18)]{Schneider2014} that
  \begin{displaymath}
    V_p(L,M)=\frac{1}{d}\int_{\SS^{d-1}}h(M,u)^p S_{p,0}(L,du),
  \end{displaymath}
  where $S_{p,0}(L,\cdot)$ is a measure on the unit sphere.  It follows
  from Fubini's theorem that
  \begin{align*}
    \E V_p(L,Z)&=\E\left(\frac{1}{d}\int_{\SS^{d-1}}h(Z,u)^p S_{p,0}(L,du)\right) \\
    &=\frac{1}{d}\int_{\SS^{d-1}}\E h(Z,u)^{p} S_{p,0}(L,du) \\
    &=\E(\sup_{0\leq t \leq 1}
    R(t))^p\frac{1}{d}\int_{\SS^{d-1}}h(K,u)^p S_{p,0}(L,du) \\
    &=\E(\sup_{0\leq t \leq 1} R(t))^p V_p(L,K). \qquad \qedhere
  \end{align*}
\end{proof}

\begin{theorem}
  \label{thr:mixed-volume2}
  Let $L$ be a convex body in $\R^d$ with the origin as an interior
  point and $X(t)$, $t\geq0$, be the standard Brownian Motion in $\R^d$.  Then,
  for all $p\in[1,\infty)$,
  \begin{displaymath}
    \E V_p(L,Z)=\frac{2^{p/2}}{\sqrt{\pi}}\Gamma
    \left(\frac{p+1}{2}\right) V_p(L,B^d).
  \end{displaymath}
\end{theorem}
\begin{proof}
  It is obvious that the finite-dimensional distributions of the process 
  $\langle X(t),u\rangle$, $t\geq0$, coincide with those of the process
  $h(K,u)R(t)$, where $R(t)=\sqrt{2}W(t)$ for the
  Brownian motion $W$ in $\R$.  Then, 
  for all $p\geq1$, 
  \begin{align*}
    \E(\sup_{0\leq t \leq 1} R(t))^p 
    =\E(\sup_{0\leq t \leq 1} \sqrt{2}W(t))^p 
    =\frac{2^p}{\sqrt{\pi}}\Gamma\left(\frac{p+1}{2}\right), 
  \end{align*}
  so that the result follows by repeating the arguments from the proof
  of Theorem~\ref{thr:mixed-volume1} and noticing that the associated
  zonoid of $X$ is $K=B^{d}/\sqrt{2}$.
\end{proof}

\section{Limit theorems for scaled convex hull}
\label{sec:limit-theorems}

Recall that $T_1$ denotes the first exit time of the process $X$ from
the unit ball. It is obvious that $\E T_1 \leq \E\min\{j\geq1 : \|
X(j)\|>1\}<\infty$, see \cite[Th.~1]{Pruitt1981}.

The standard (Skorohod) metric on the space $D([0,1],\R^d)$ of all
right-continuous $\R^d$-valued functions with left limits defined on
$[0,1]$ is defined by
\begin{displaymath}
  d_{J_1}(f(t),g(t))=\inf_{\lambda\in\Lambda}\{\| f(\lambda(t))-g(t)\|_{\infty}+
  \| \lambda(t)-t\|_{\infty}\},
\end{displaymath}
where $\Lambda$ is the set of strictly increasing functions $\lambda$
mapping $[0,1]$ onto itself, such that both $\lambda$ and its inverse
$\lambda^{-1}$ are continuous.

\begin{lemma}
  \label{lemma:continuity}
  The map that associates with a function $f\in D([0,1],\R^d)$ the
  closed convex hull of its range $\{f(t):\; 0\leq t\leq1\}$ is continuous in
  the Hausdorff metric on $\mathcal{K}_d$. 
\end{lemma}
\begin{proof}
  It suffices to note that the Hausdorff distance between the closed ranges
  of $f,g\in D([0,1],\R^d)$ is dominated by $d_{J_1}(f,g)$, and the
  Hausdorff distance between the convex hulls of compact sets equals
  the Hausdorff distance between the sets themselves.
\end{proof}

\begin{theorem}
  \label{thr:lt}
  Assume that $X(t)$, $t\geq0$, is a L\'evy process in $\R^d$ such
  that $X(T_1)$ lies in the domain of attraction of a strictly
  $\alpha$-stable random vector $\eta$, that is the sum of $n$
  i.i.d. copies of $X(T_1)$ scaled by $(n^{1/\alpha}\ell(n))^{-1}$ with a
  slowly varying function $\ell$ converges in distribution to $\eta$.
  Then $(t^{1/\alpha}\ell(t))^{-1} Z_t$ converges in distribution to
  $\conv\{Y(s) : 0\leq s \leq (\E T_1)^{-1}\}$, where $Y$ is a L\'evy
  process such that $Y(1)$ coincides in distribution with $\eta$.
\end{theorem}
\begin{proof}
  Without loss of generality assume that $\ell(n)=1$ for all $n$. 
  As in the proof of Theorem~\ref{thr:integrability}, split the path
  of the L\'evy process into several parts using the stopping times
  defined in \eqref{eq:Ti} and consider the renewal process $N_s$.
  Then $S_n=X(T_n)$, $n\geq1$, forms a random walk embedded in the
  process  and 
  \begin{displaymath}
    \conv\{0,S_1,\dots,S_{N_t}\}\subseteq Z_t \subseteq 
    \conv\{0,S_1,\dots,S_{N_t}\} +B^d
  \end{displaymath}
  for all $t>0$, so that
  \begin{displaymath}
    \frac{\conv\{0,S_1,\dots,S_{N_t}\}}{N_t^{1/\alpha}}
    \left(\frac{N_t}{t}\right)^{1/\alpha}
    \subseteq \frac{Z_t}{t^{1/\alpha}}
    \subseteq \frac{\conv\{0,S_1,\dots,S_{N_t}\}}{N_t^{1/\alpha}}
    \left(\frac{N_t}{t}\right)^{1/\alpha}
    +\frac{B^d}{t^{1/\alpha}}.
  \end{displaymath}  
  In the square integrable case, the convergence of
  $n^{-1/2}\conv\{0,S_1,\dots,S_{n}\}$ to the convex hull of the
  Brownian motion is established in \cite[Th.~2.5]{Wade2015}. In the
  general case, Donsker's invariance principle with stable limits
  \cite[Th.~4.5.3]{Whitt2002} implies that $n^{-1/\alpha} S_{\lfloor
    nt\rfloor}$, $0\leq t\leq 1$, converges weakly to the process
  $Y(t)$, $0\leq t \leq1$, in $(D,d_{J_1})$. Using the single
  probability space argument, the convergence holds if $n$ is replaced
  by $N_t$ and $t\to\infty$. Lemma~\ref{lemma:continuity} and
  the continuous mapping
  theorem yield the weak convergence of
  $N_t^{-1/\alpha}\conv\{0,S_1,\dots,S_{N_t}\}$ to $\conv\{Y(t):\; 0\leq
  t\leq 1\}$ in $\mathcal{K}_d$ with the Hausdorff metric.

  The final statement follows from 
  $N_t/t\to (\E T_1)^{-1}$ a.s. as $t\to\infty$,
  taking into account the scaling property of the process $Y$ and the
  fact that $t^{-1/\alpha}B^d \to\{0\}$ as $t\to\infty$. 
\end{proof}

Recall that $\nu$ denotes the L\'evy measure of $X$. 

\begin{theorem}
  \label{thr:lm}
  \begin{itemize}
  \item[(i)] If 
    \begin{equation}
      \label{eq:sm}
      \int_{\|x\|>1} \|x\|^2\nu(dx)<\infty,
    \end{equation}
    then $X(T_1)$ belongs to the domain of attraction of the Gaussian
    law.
  \item[(ii)] Assume that \eqref{eq:sm} does not hold, $\nu(\{x:\;
    \|x\|\geq r\})$, $r>0$, is a regularly varying function at infinity 
    with exponent $-\alpha$ for $\alpha\in(0,2)$, and
    \begin{displaymath}
      \frac{\nu(\{x:\; \|x\|\geq r, x/\|x\|\in A\})}
      {\nu(\{x:\; \|x\|\geq r\})} \to 
      \frac{\Lambda(A)}{\Lambda(\SS^{d-1})} \quad \text{as }\; r\to\infty
    \end{displaymath}
    for all Borel subsets $A$ of the unit sphere
    $\SS^{d-1}$ such that $\Lambda(\partial A)=0$, where $\Lambda$ is a
    finite Borel measure on the unit sphere which is not supported by
    any proper linear subspace. Then $X(T_1)$ belongs to the domain of
    attraction of an $\alpha$-stable law.
  \end{itemize}
\end{theorem}
\begin{proof}
  It is well known that a L\'evy process $X$ can be expressed as the
  sum of three independent L\'evy processes $X_1, X_2, X_3$, where
  $X_1$ is a linear transform of a Brownian motion with drift,
  $X_2$ is a compound Poisson process having only jumps of norm
  strictly larger than $2$ and $X_3$ is a pure jump process having jumps of size
  at most $2$.  This decomposition is known as the L\'evy-It\^{o}
  decomposition.  It is obvious that
  \begin{displaymath}
    X(T_1)=X(T_1-)+X_2(T_1)-X_2(T_1-)+X_3(T_1)-X_3(T_1-).
  \end{displaymath}
  Since the norm of $X(T_1-)+X_3(T_1)-X_3(T_1-)$ is at most 3, it is
  in the domain of attraction of the normal distribution.  
  
  The compound Poisson process 
  \begin{displaymath}
    X_2(t)=\sum_{i:\; \tau_i\leq t}\xi_i
  \end{displaymath}
  is given by the sum of i.i.d. random vectors $\xi_i$, $i\geq1$, with
  the distribution given by $\nu$ restricted onto the complement of
  the ball of radius $2$ and normalised to become a probability
  measure. Here $\{\tau_i,i\geq1\}$ is a homogeneous Poisson process
  on $\R_+$. Thus, $T_1\leq \tau_1$ a.s., and $X_2(T_1)-X_2(T_1-)$ is
  zero if $T_1<\tau_1$, while otherwise it equals $\xi_1$. If
  (\ref{eq:sm}) holds, then 
  $\E\|\xi_1\|^2<\infty$, so that $X_2(T_1)-X_2(T_1-)$ belongs to the
  domain of attraction of a normal law. Otherwise, the conditions of
  the theorem guarantee that $X_2(T_1)-X_2(T_1-)$ belongs to the
  domain of attraction of an $\alpha$-stable law, see
  \cite[Th.~8.2.18]{Meerschaert2001}.  Then $X(T_1)$ also belongs to
  the domain of attraction of the same law.
\end{proof}

%%%%%%%%%%%%%%%%%%%%%%%%%%%%%%%%%%%%%%%%%%%%%%%%%%%%%%%%%%%%%%%%%%%
%%                                                               %%
%% Use the two commands below for producing your bibliography    %%
%% with bibtex, then comment again the commands and include the  %%
%% content of the .bbl file in this file below the commands.     %%
%%                                                               %%
%%%%%%%%%%%%%%%%%%%%%%%%%%%%%%%%%%%%%%%%%%%%%%%%%%%%%%%%%%%%%%%%%%%

\bibliographystyle{amsplain}
%\bibliography{references}
% add below the content of your .bbl file produced by bibtex.

\providecommand{\bysame}{\leavevmode\hbox to3em{\hrulefill}\thinspace}
\providecommand{\MR}{\relax\ifhmode\unskip\space\fi MR }
% \MRhref is called by the amsart/book/proc definition of \MR.
\providecommand{\MRhref}[2]{%
  \href{http://www.ams.org/mathscinet-getitem?mr=#1}{#2}
}
\providecommand{\href}[2]{#2}

%%%%%%%%%%%%%%%%%%%%%%%%%%%%%%%%%%%%%%%%%%%%%%%%%%%%%%%%%%%%%%%%%%%
%%                                                               %%
%% You may add acknowledgments (optional).                       %%
%%                                                               %%
%%%%%%%%%%%%%%%%%%%%%%%%%%%%%%%%%%%%%%%%%%%%%%%%%%%%%%%%%%%%%%%%%%%

\ACKNO{IM is grateful to Gennady Samorodnitsky for the idea of path
  splitting used in the proof of Theorem~\ref{thr:integrability} and
  comments on the draft of this work. Bojan Basrak helped to come up
  with the current variant of Theorem~\ref{thr:interior-point}. The
  authors are indebted to the referee who suggested improvements to
  the proofs of several results.}

%%%%%%%%%%%%%%%%%%%%%%%%%%%%%%%%%%%%%%%%%%%%%%%%%%%%%%%%%%%%%%%%%%%
%%                                                               %%
%% You have reached the end of your document.                    %%
%%                                                               %%
%%%%%%%%%%%%%%%%%%%%%%%%%%%%%%%%%%%%%%%%%%%%%%%%%%%%%%%%%%%%%%%%%%%

\end{document}